\documentclass[10pt]{amsart}
\usepackage{amssymb,amsthm,amsmath,amsfonts}
\usepackage{natbib}
\usepackage{hyperref}
\usepackage{enumerate}

\makeatletter
\g@addto@macro\th@plain{\thm@headpunct{}}
\makeatother

\newtheorem{thm}{Theorem}[section]

\newtheorem{rem}[thm]{Remark}

\def \R {{\mathbb R}}
\def\N{{\mathbb{N}}}
\def\Z{{\mathbb{Z}}}
\def \P {{\mathbb P}}
\def \E {{\mathbb E}}
\def \Q {{\mathbb Q}}

\newcommand{\dd}{\mathrm{d}}

\newcommand{\Ind}{\mathbf{1}_}

\title{The left tail of renewal measure}

\author{Bartosz Ko\l{}odziejek}
\address{Faculty of Mathematics and Information Science\\Warsaw University of Technology\\ Koszykowa 75\\00-662 Warsaw, Poland}
\email{kolodziejekb@mini.pw.edu.pl}

\subjclass[2010]{60K05}

\keywords{renewal theory; renewal measure.}

\begin{document}

\begin{abstract}
In the paper, we find exact asymptotics of the left tail of renewal measure for a broad class of two-sided random walks. {We only require that an exponential moment of the left tail is finite.}
Through a simple change of measure approach, our result turns out to be almost equivalent to Blackwell's Theorem.
\end{abstract}

\maketitle

\section{Introduction}
Let $(X_k)_{k\geq 1}$ be a sequence of independent copies of a random variable $X$ with $\E X>0$ (we allow $\E X=\infty$). 
Further, define $S_n=X_1+\ldots+X_n$, $n\geq 1$ and $S_0=0$.
The measure defined by 
$$H(B):=\sum_{n=0}^\infty \P(S_n\in B),\qquad B\in\mathcal{B}(\R)$$
is called the \emph{renewal measure of $(S_n)_{n\geq 1}$}. 

We say that the distribution of a random variable $X$ is \emph{$d$-arithmetic} ($d>0$) if it is concentrated on $d\Z$ and not concentrated on $d' \Z$ for any $d'>d$. A distribution is said to be \emph{non-arithmetic} if it is not $d$-arithmetic for any $d>0$.

A fundamental result of renewal theory is the Blackwell Theorem (\cite{Black53}): if the distribution of $X$ is non-arithmetic, then for any $h>0$,
\begin{align}\label{eq20}
H\left( (x,x+h]\right)\longrightarrow\frac{h}{\E X}\qquad\mbox{ as }x\to\infty.
\end{align}
If the distribution of $X$ is $d$-arithmetic, then for any $h>0$,
\begin{align}\label{eq10}
H\left( (d n,d n+h]\right)\longrightarrow\frac{d \left\lfloor h/d\right\rfloor}{\E X}\qquad\mbox{ as }n\to\infty.
\end{align}
The above results remain true if $\E X=\infty$ with the usual convention that $c/\infty=0$ for any finite $c$.
In the infinite-mean case the exact asymptotics of $H((x,x+h])$ are also known. 
Assume that $X$ is a non-negative random variable with a non-arithmetic law such that $\P(X>x)=L(x)x^{-\alpha}$ with $\alpha\in(0,1)$, where $L$ is a slowly varying function. Then $\E X=\infty$.
If $\alpha\in(1/2,1)$, then without additional assumptions the so called Strong Renewal Theorem holds, for $h>0$,
\begin{align}\label{SRT}
m(x) H\left( (x,x+h]\right)\longrightarrow\frac{h}{\Gamma(\alpha)\Gamma(2-\alpha)}\qquad \mbox{ as }x\to\infty,
\end{align}
where $m(x)=\int_0^x \P(X>t)\mathrm{d}t\sim L(x)x^{1-\alpha}/(1-\alpha)\to\infty$. 
Here and later on $f(x)\sim g(x)$ means that $f(x)/g(x)\to 1$ as $x\to\infty$.

The case of $\alpha\in(0,1/2]$ is much harder and was completely solved just recently by \cite{CD16}. It was shown that if $\alpha\in(0,1/2]$ and $X$ is a non-negative random variable with regularly varying tail, then \eqref{SRT} holds if and only if (\cite[Proposition 1.11]{CD16})
\begin{align}\label{SRTc}
\lim_{\delta\to0}\limsup_{x\to\infty}\int_1^{\delta x} \frac{F(x)-F(x-z)}{\overline{F}(z)z^2}\dd z=0,
\end{align}
where $F$ is the cumulative distribution function of $X$ and $\overline{F}=1-F$. It was already observed by \cite[Theorem 3.1]{Kev16} that this result generalizes to $X$ attaining negative values as well if additionally \begin{align}\label{eq2}\P(X\leq -x)=o(e^{-r x})\qquad \mbox{ as }x\to\infty\end{align}
for some $r>0$. This will be our setup. Full picture of SRT for random walks is also known (\cite[Theorem 1.12]{CD16}).

It is clear that $\lim_{x\to\infty}H((-\infty,-x))=0$.
There are considerably fewer papers dedicated to analysis of exact asymptotics of such object than of $H((x,x+h])$ as in Blackwell's Theorem. Under some additional assumptions we know more about the asymptotic behaviour of the left tail. \cite{Stone65} proved that 
if for some $r>0$ \eqref{eq2} holds, then for some $r_1>0$, 
\begin{align}\label{Hleft}
H((-\infty,-x))=o(e^{-r_1 x})\qquad\mbox{ as }x\to\infty.
\end{align}
Stone's result was strengthened 
by \cite{Gen69}, where exact asymptotics as well the speed of convergence of the remainder term are given for $d$-arithmetic and spread-out laws (i.e. laws, whose $n$th convolution has a nontrivial absolutely continuous part for some $n\in\mathbb{N}$).
An important contribution regarding the asymptotics of the left tail of renewal measure was made by \cite{Carl83}, who concerned with the case when $\E|X|^m<\infty$ for some $m\geq 2$, but this does not fit well into our setup. We allow $\E X_+=\infty$, but on the other hand we require that some exponential moments of $X_-$ exist.
The results mentioned above were obtained using some analytical methods, whereas we will use a simple probabilistic argument, which boils down the asymptotics of $H((-\infty,-x))$ to the asymptotics of $\widetilde{H}((x,x+h])$, where $\widetilde{H}$ is some new (possibly defective, see below) renewal measure.

\subsection{Defective renewal measure}
For $\rho\in(0,1)$ consider
$$H_\rho(B):=\sum_{n=0}^\infty \rho^n\P(S_n\in B),\qquad B\in\mathcal{B}(\R),$$
where $(S_n)_{n\geq 1}$ is, as in the previous section, a random walk starting from $0$.
$H_\rho$ is called a \emph{defective renewal measure of $(S_n)_{n\geq 1}$}. In contrast to the renewal measure, $H_\rho$ is a finite measure.
Let $\tau$ be independent of $(S_n)_{n\geq 1}$ and $\P(\tau=n)=(1-\rho)\rho^n$, $n=0,1,\ldots$.
Then $H_\rho(B)=\P(S_\tau\in B)/(1-\rho)$.
It is well known that if the distribution of $S_1$ is subexponential, then $\P(S_\tau>x)\sim \E\tau\P(S_1>x)$.
Here, we are interested in exact asymptotics of $H_\rho(B)$ when $B=(x,x+T]$ for any $T>0$. In this context, local subexponentiality is the key concept (\cite{AFK03}).

Let $\mu$ be a probability measure on $\R$. For $T>0$ we write $\Delta=(0,T]$ and $x+\Delta=(x,x+T]$.
We say that $\mu$ belongs to the class $\mathcal{L}_{\Delta}$ if $\mu(x+\Delta)>0$ for sufficiently large $x$ and
\begin{align}\label{Lloc}
\frac{\mu(x+s+\Delta)}{\mu(x+\Delta)}\to 1\quad\mbox{ as }x\to\infty,
\end{align}
uniformly in $s\in[0,1]$.

We say that $\mu$ is $\Delta$-subexponential if $F\in \mathcal{L}_\Delta$ and
$$\mu^{\ast2}(x+\Delta)\sim 2 \mu(x+\Delta).$$
Then we write $\mu\in\mathcal{S}_{\Delta}$.
Finally, $\mu$ is called \emph{locally subexponential} if $\mu\in \mathcal{S}_{\Delta}$ for any $T>0$. {We denote this class by $\mathcal{S}_{loc}$.}

The following Theorem is an obvious conclusion from \cite[Theorem 1.1]{WY09}.
\begin{thm}\label{loc}
Assume that $\mu$ is a probability measure on $\R$ such that
$$\int_{\R}e^{-\varepsilon x}\mu(\dd x)<\infty\qquad\mbox{for some }\varepsilon>0.$$
For $0<\rho<1$ define
$$\eta=\sum_{n=0}^\infty \rho^n \mu^{\ast n}.$$
Then $\mu\in \mathcal{S}_{\Delta}$ if and only if $\eta/(1-\rho)\in \mathcal{S}_{\Delta}$ if and only if
$$\eta(x+\Delta)\sim \frac{\rho}{(1-\rho)^2} \mu(x+\Delta).$$
\end{thm}

Some examples of measures from $\mathcal{S}_{loc}$ may be found in \cite[Section 4]{AFK03}.

\section{Main result}
Assume that $X$ is a random variable with $\E X\in(0,\infty]$.
We define the Laplace transform of the distribution of $X$ by
$$g(\theta):=\E e^{-\theta X}.$$
{Function $g$ is convex and lower-semicontinuous.} We are interested in a situation of an exponentially decaying left tail, that is,
\begin{align}\label{exp}
{\P(X<0)>0\qquad\mbox{ and }\qquad}g(\theta)<\infty\qquad\mbox{for some }\theta>0.
\end{align}
Under \eqref{exp} we define 
\begin{align}\label{def}
\kappa:=\sup\{ \theta>0\colon g(\theta)<1\}\qquad\mbox{ and } \qquad \rho:=g(\kappa).
\end{align}
{Since $g'(0)=-\E X<0$, $\kappa$ is strictly positive. Moreover, we have $g(\theta)\to\infty$ as $\theta\to\infty$ and thus $\kappa$ is finite. In general we have $0<\rho\leq 1$ and a sufficient condition for }$\rho=1$ is that $g(\theta)<\infty$ for all $\theta>0$.

\begin{thm}\label{mainth}
Assume $X$ is a random variable with a positive (possibly infinite) expectation such that \eqref{exp} holds.
Let $H$ be the renewal measure of $(S_n)_{n\geq 0}$, where $S_n=\sum_{k=1}^n X_k$ for $n\in\mathbb{N}$, $S_0=0$ and $X_k$ are independent copies of $X$.
Define $\kappa$ and $\rho$ as in \eqref{def}.

\begin{itemize}
\item[\rm (a)] Assume that $\rho=1$.
\begin{itemize}
\item[\rm (a-i)] Assume that $X$ has a non-arithmetic distribution. Then
$$\lim_{x\to\infty}e^{\kappa x}H((-\infty,-x))=\frac{1}{\kappa g'(\kappa)}\in[0,\infty).$$
Moreover, if 
\begin{align}\label{eq5}
\E e^{-\kappa X}\Ind{\{-X>t\}}\sim\frac{L(t)}{t^\alpha}
\end{align}
for some $\alpha\in(0,1)$ and a slowly varying function $L$, then $g'(\kappa)=\infty$. 
For $\alpha\in(0,1/2]$, assume additionally that $F(t)=\E e^{-\kappa X}\Ind{\{-X\leq t\}}$ satisfies \eqref{SRTc}. 
In such case, 
$$e^{\kappa x}H((-\infty,-x))\sim\frac{1}{\Gamma(\alpha)\Gamma(2-\alpha)}\frac{1}{\kappa m(x)},$$
where $m(x)\sim  L(x)x^{1-\alpha}/(1-\alpha)$.

\item[\rm (a-ii)] Assume that $X$ has a $d$-arithmetic distribution. Then
$$\lim_{n\to\infty}e^{\kappa d n}H((-\infty,-n d))=\frac{d}{(e^{\kappa d}-1) g'(\kappa)}.$$
\end{itemize}

\item[\rm (b)] If $\rho<1$, then
\begin{align}\label{eq50}
\lim_{x\to\infty}e^{\kappa x}H((-\infty,-x))=0.
\end{align}
Moreover, if $\rho^{-1}\E e^{-\kappa X}\Ind{\{-X\in \cdot\}}\in \mathcal{S}_{loc}$, then
$$e^{\kappa x}H((-\infty,-x))\sim \frac{\E e^{-\kappa X}\Ind{x<-X\leq x+1}}{\kappa (1-\rho)^2}.$$
\end{itemize}
\end{thm}
\begin{rem}
\rm{ }Condition \eqref{eq5} is implied by $$\P(-X>t)= \frac{\alpha}{\kappa}\frac{L(t)}{t^{\alpha+1}}e^{-\kappa t},\qquad t>0.$$
Indeed, for any slowly varying function $L$ and $\beta<-1$, \cite[Proposition 1.5.10]{BGT89} asserts that
$$\int_x^\infty t^\beta L(t)\mathrm{d}t\sim x^{\beta+1}L(x)/(-\beta-1).$$
\end{rem}
\begin{rem}
\rm{ }Under the same assumptions, a stronger result concerning \rm{(a-ii)} is proved in \cite[Theorem 2]{Gen69} (the remainder term is also exponential).
\end{rem}
\begin{rem}
\rm{ } If $X$ has a non-arithmetic distribution, for any $\delta>0$, we obtain ``more local'' behaviour: $$\lim_{x\to\infty}e^{\kappa x}H((-x-\delta,-x))=\frac{1-e^{-\delta\kappa}}{\kappa g'(\kappa)}.$$
\end{rem}
\begin{proof}[Proof of Theorem~\ref{mainth}]
Note that $g'(\kappa)=-\E X e^{-\theta X}$ is positive ($1=g(0)=g(\kappa)$ and $g$ is convex), but may infinite.

Define $\mathcal{F}_n=\sigma(X_1,\ldots,X_n)$ and let $\mathcal{F}_\infty$ be the smallest $\sigma$-field containing all $\mathcal{F}_n$.  
On $(\Omega,\mathcal{F}_\infty)$ we define a new measure $\Q$ via projections
$$\Q((X_1,\ldots,X_n)\in B)={\rho^{-n}}\E e^{-\kappa S_n} \Ind{\{-(X_1,\ldots,X_n)\in B\}},\qquad  B\in\mathcal{B}(\mathbb{R}^n),$$
where $S_n=X_1+\ldots+X_n$, $n\in\N$ and $S_0=0$.
By the definition of $\rho$, $\Q$ is a probability measure. Moreover, $(X_n)_{n\geq 1}$ is an iid sequence under $\Q$ as well.
Let $\E_{\Q}$ denote the corresponding expectation. 
For any Borel function $f\colon \R\to\R_+$ one has
$$\E f(S_n)={\rho^n}\E_{\Q}e^{-\kappa S_n}f(-S_n).$$
Thus,
$$\P(S_n< -x)={\rho^n}\E_{\Q}e^{-\kappa S_n}\Ind{S_n>x}={\rho^n}\int_{(x,\infty)} e^{-\kappa t} \Q_X^{\ast n}(\mathrm{d}t).$$
Moreover, observe that $\E_{\Q} X=-\E X e^{-\kappa X}=g'(\kappa){\in(0,\infty]}$, thus $(S_n)_n$ has a positive drift under $\Q$ as well. Hence, for $x>0$,
$$H((-\infty,-x))=H_{\P}((-\infty,-x))=\sum_{n=1}^\infty \P(S_n< -x)=\int_{(x,\infty)} e^{-\kappa t} H_{\Q}(\dd t),$$
where $H_{\Q}=\sum_{n=0}^\infty {\rho^n}\Q_X^{\ast n}$ is the (defective if $\rho<1$) renewal measure of $(S_n) _{n\geq 0}$ under $\Q$.

Writing $e^{-\kappa t}=\kappa \int_t^\infty e^{-\kappa s}\dd s$, through Tonelli's Theorem, we arrive at key identity:
\begin{align}\label{keyeq}
H_{\P}((-\infty,-x))=\kappa \int_x^\infty e^{-\kappa s}H_{\Q}((x,s])\dd s=
\kappa e^{-\kappa x}\int_0^\infty e^{-\kappa h}H_{\Q}((x,x+h])\dd h.
\end{align}

Consider first the case of $\rho=1$.
For any renewal measure $H$ we have $H((x,x+h])\leq \alpha h+\beta$ for some $\alpha,\beta>0$ and all $x$, thus by Lebesgue's Dominated Convergence Theorem and \eqref{eq20},
$$\lim_{x\to\infty}e^{\kappa x}H_{\P}((-\infty,-x))=
\kappa \int_0^\infty e^{-\kappa h} \lim_{x\to\infty }H_{\Q}((x,x+h])\dd h=\frac{1}{\kappa \E_{\Q}X},$$
which gives the first part of \rm{(a-i)}. For \rm{(a-ii)} use \eqref{eq10}, instead of \eqref{eq20}.

For the second part of \rm{(a-i)}, observe that 
$$\Q(X>t)=\E e^{-\kappa X}\Ind{\{-X>t\}}=\frac{L(t)}{t^\alpha},$$
thus the result follows by the Strong Renewal Theorem.

If $\rho<1$, then $H_\Q$ is a finite measure and \eqref{eq50} follows again by Lebesgue's Dominated Convergence Theorem.

Consider now the case, when $\Q_X\in\mathcal{S}_{loc}$. Since $\E_\Q e^{-\kappa X}=1<\infty$, by Theorem \ref{loc}, we have \mbox{$(1-\rho)H_{\Q}\in \mathcal{S}_{loc}$} and
\begin{align}\label{eqHQ}
H_{\Q}((x,x+{1}])\sim\frac{\rho}{(1-\rho)^2}\Q_X((x,x+{1}]).
\end{align}
Define $L(y):=H_{\Q}((\log y,\log y+1])$. By \eqref{Lloc}, $L$ is a slowly varying function. 
Moreover, for any $h>0$ 
$$\frac{H_{\Q}((x,x+h])}{H_{\Q}((x,x+1])}\leq  {\sum_{n=1}^{\lceil h \rceil}\frac{ H_{\Q}((x+n-1,x+n])}{H_{\Q}((x,x+1])}=}\sum_{n=1}^{\lceil h\rceil} \frac{L(e^{x+n-1})}{L(e^x)}.$$
By Potter bounds (\cite[Theorem 1.5.6]{BGT89}) for any $\varepsilon>0$ and $C>1$ there exists $x_0$ such that for $x>x_0$,
$${\frac{H_{\Q}((x,x+h])}{H_{\Q}((x,x+1])}\leq}\sum_{n=1}^{\lceil h\rceil} C e^{\varepsilon (n-1)}\leq C \lceil h\rceil e^{\varepsilon \lceil h\rceil}.$$
Observe that \eqref{Lloc} implies for $h>0$,
$$H_{\Q}((x,x+h])\sim h\,H_{\Q}((x,x+1]).$$
{Indeed, for $h=k/n\in\Q_+$ one gets
$$H_{\Q}((x,x+\tfrac{k}{n}])=\sum_{i=1}^k H_{\Q}((x+\tfrac{i-1}{n},x+\tfrac{i}{n}])\sim k\, H_{\Q}((x,x+\tfrac{1}{n}])$$
and
$$H_{\Q}((x,x+\tfrac{1}{n}])\sim \frac{1}{n}\sum_{i=1}^n H_\Q((x+\tfrac{i-1}{n},x+\tfrac{i}{n}])=\frac{1}{n}H_\Q((x,x+1]).$$}
By monotonicity, $f(h):=\lim_{x\to\infty} H_{\Q}((x,x+h])/H_{\Q}((x,x+1])$ exists for all $h>0$ and $f(h)=h$.
{Thus, by \eqref{keyeq} and Lebesgue's Dominated Convergence Theorem we conclude that
\begin{align*}
\lim_{x\to\infty} e^{\kappa x}\frac{ H_{\P}((-\infty,-x))} {H_{\Q}((x,x+1])}= \kappa \int_0^\infty e^{-\kappa h}\lim_{x\to\infty}\frac{H_{\Q}((x,x+h])}{H_{\Q}((x,x+1])}\dd h=1/\kappa.
\end{align*}
The use of \eqref{eqHQ} completes the proof.}
\end{proof}

\subsection*{Acknowledgements}
The author was partially supported by NCN Grant No. 2015/19/D/ST1/03107. 
I am grateful to the referee for a number of helpful suggestions for improvement in the article.

\bibliographystyle{plainnat}


\def\polhk#1{\setbox0=\hbox{#1}{\ooalign{\hidewidth
  \lower1.5ex\hbox{`}\hidewidth\crcr\unhbox0}}}

\end{document}